\documentclass[10pt]{amsart}
\usepackage{amsthm,hyperref}
\usepackage{amsmath}
\usepackage{amsfonts}
\usepackage{amssymb}
\usepackage[all]{xy}
\xyoption{line}
\usepackage{enumitem}
\usepackage{graphicx}
\usepackage{amssymb,amsmath,amsfonts,amscd,euscript}
\usepackage{tikz}
\usetikzlibrary{matrix,arrows}
\usetikzlibrary{patterns}
\usetikzlibrary{shapes,decorations,shadows}

\newcommand{\QQ}{\mathcal{Q}}
\newcommand{\A}{\mathcal{A}}
\newcommand{\al}{\alpha}
\newcommand{\ee}{\varepsilon}
\DeclareMathOperator{\Rep}{Rep}
\newcommand{\degg}{\leq_{\deg}}

\DeclareMathOperator{\GL}{GL}

\DeclareMathOperator{\SO}{SO}
\DeclareMathOperator{\B}{B}

\DeclareMathOperator{\G}{G}
\DeclareMathOperator{\OO}{O}
\DeclareMathOperator{\SP}{SP}

\DeclareMathOperator{\Hom}{Hom}
\DeclareMathOperator{\Lie}{Lie}
\DeclareMathOperator{\Ext}{Ext}

\newcommand{\dimv}{{\bf dim}\,}
\newcommand{\form}{\langle-,-\rangle}

\numberwithin{equation}{section}
\newtheorem{theorem}{Theorem}[section]
\newtheorem{lemma}[theorem]{Lemma}
\newtheorem{proposition}[theorem]{Proposition}
\newtheorem{corollary}[theorem]{Corollary}
\newtheorem{conjecture}[theorem]{Conjecture}
\theoremstyle{definition}
\newtheorem{remark}[theorem]{Remark}
\newtheorem{definition}[theorem]{Defintion}
\newtheorem{example}[theorem]{Example}

\newtheorem{question}[theorem]{Question}

\begin{document}

\title[Symmetric degenerations are not induced]{Symmetric degenerations are not in general induced by type A degenerations}
\author{Magdalena Boos}
\address{Magdalena Boos. Ruhr University Bochum, Faculty of Mathematics,  44780 Bochum, (Germany).} 
\email{magdalena.boos-math@rub.de}
\author{Giovanni Cerulli Irelli}
\address{Giovanni Cerulli Irelli. Sapienza-Universit\`a di Roma. Department S.B.A.I.. Via Scarpa 10, 00164, Roma (Italy)}
\email{giovanni.cerulliirelli@uniroma1.it}

\begin{abstract}
We consider a symmetric quiver with relations. Its (symmetric) representations of a fixed symmetric dimension vector are encoded in the (symmetric) representation varieties. The orbits by a (symmetric) base change group action are the isomorphism classes of (symmetric) representations. The symmetric orbits are induced by simply restricting the non-symmetric orbits. However, when it comes to orbit closure relations, it is so far an open question under which assumptions they are induced. In connection with Borel orbits of $2$-nilpotent matrices of classical Lie algebras, we describe an explicit example of a quiver of finite representation type for which orbit closure relations are induced in types B and C, but not in type D. 
\end{abstract}
\maketitle

\section{Introduction}\label{Sec:Intro}

Let $\A=\mathrm{k}\QQ/I$ be a symmetric quiver algebra over the field $\mathrm{k}$ of complex numbers. We fix a $\QQ_0$-graded $\mathrm{k}$-vector space $V$ and denote the representation variety of representations with underlying vector space $V$ by $R(\A,V)$. Inside of $R(\A,V)$ there is a subvariety $R(\A,V)^{\form, \ee}$ of so-called $\ee$-representations; here $\ee$ is a sign and $\form$ is a non-degenerate bilinear form on $V$ (see Subsection~\ref{Subsec:SymmRep} for the definitions). An $\ee$-representation is a symmetric representation which has an orthogonal or a symplectic structure. There are natural (symmetric) base change actions on these varieties; their orbits correspond to isomorphism classes of (symmetric) representations. It is natural to ask, whether the orbits and their closures can be translated easily between the two group actions. Concerning the orbits, it is known (\cite[Theorem~2.5]{BCI}, \cite[Theorem~2.6]{DW}, \cite[Section~2.1]{MWZ}) that a symmetric orbit is  induced by a single non-symmetric orbit, i.e. the restriction of an orbit $\mathcal{O}\subseteq R(\A,V)$ to $R(\A,V)^{\form, \ee}$ is a symmetric orbit. The question under which assumptions the orbit closure relations are induced by restricting in the same way, is still open. In this article, we provide a negative answer to this question by producing an  example of a quiver algebra of finite representation type, for which symmetric orbit closures are not induced.

This work is heavily based on our preceding article \cite{BCI} where many details on the symmetric representation theory of a symmetric quiver algebra can be found. In particular, in said article we prove that for $\QQ$ a Dynkin quiver with self-duality, orbit closure relations are induced.

We structure this article as follows: The setup of our Main Question \ref{Question} is explained in Section \ref{Sec:Setup} where aforesaid  question is posed.  At the same time we recall some general knowledge on algebras with self-dualities. In Section \ref{Sec:orders}, we define several partial orders which embed our Main Question into a representation-theoretical and homological frame. Our counterexample is described in Section \ref{Sec:Seesaw} where we look at the so-called Seesaw algebra and define particular symmetric representations which degenerate in type A, but not in type D. In Section~\ref{Rem:LieTh} we discuss the relationship between the Seesaw algebra and the Borel orbits of $2$-nilpotent elements in classical Lie algebras. In said section we also provide another proof of our counterexample in the language of \cite{GMP}. We end the paper by posing conjectures which are likely to hold true from our current perspective on the topic.

\section{Setup}\label{Sec:Setup}

Let $\mathrm{k}=\textbf{C}$ be the field of complex numbers and let  $\QQ=(\QQ_0,\QQ_1,s,t)$ be a  finite \emph{quiver}, that is,  an oriented graph with a finite set of \emph{vertices} $\QQ_0$, a finite set of  \emph{edges} $\QQ_1$ and two maps $s,t:\QQ_1\rightarrow \QQ_0$ which provide the orientation $\alpha: s(\alpha) \rightarrow t(\alpha)$ of the edges.  Let us consider the elements of $\QQ_1$ as arrows.  A sequence of arrows $\omega=\alpha_s \cdots\alpha_1$ is called a \emph{path} in $\QQ$ whenever $t(\alpha_{i})=s(\alpha_{i+1})$ for all $i$; we formally include a path $\ee_i:i \rightarrow i$ of length zero for each $i\in \QQ_0$. The \emph{path algebra} $\mathrm{k}\QQ$ of $\QQ$ is the $\mathrm{k}$-algebra spanned as a $\mathrm{k}$-vector space by the set of all paths in $\QQ$ together with the concatenation of paths as multiplication. Let $R\subseteq \mathrm{k}\QQ$ be the $2$-sided ideal generated by all arrows in $\QQ_1$; it is called the \emph{arrow ideal}. Then every ideal $I\subseteq \mathrm{k}\QQ$  which determines an integer $s$ with $R^s\subseteq I\subseteq R^2$  is called \emph{admissible}. If $I$ is admissible, then the quotient algebra $\A:=\mathrm{k}\QQ/I$ is a finite-dimensional and associative \emph{quiver algebra} \cite{ASS}. We denote by $\Rep(\A)$ the category of finite-dimensional $\A$-modules.

Now assume that $\QQ$ comes with a symmetry as defined in \cite{DW}, that is, we consider a tuple $(\QQ,\sigma)$ where $\sigma:Q\rightarrow Q^{op}$ is an involutive bijection of $\QQ_0$ which induces an arrow-reversing involution of $\QQ_1$. The pair  $(\QQ,\sigma)$  is called a \emph{symmetric quiver}. Assume that an admissible ideal $I\subset \mathrm{k}\QQ$ fulfills $\sigma(I)= I$; then the pair $(\mathrm{k}\QQ/I,\sigma)$ is called a \emph{symmetric quiver algebra}.

\subsection{Quiver representations}

Let $\A=\mathrm{k}\QQ/I$ be a quiver algebra. Let $V=\oplus_{i\in \QQ_0}V_i$ be a finite dimensional $\QQ_0$-graded vector space of graded dimension $\mathbf{d}=\dimv V=(\dim V_i)_{i\in \QQ_0}$. We denote by $R(\A,V)$ the variety of $\A$-representations having $V$ as underlying vector space, that is, its elements are collections $f=(f_\alpha:V_{s(\alpha)}\rightarrow V_{t(\alpha)})_{\alpha\in \QQ_1}$ of linear maps such that $f_{\pi}=0$ for every $\pi\in I$. Thus,
\[
R(\A,V)\subseteq R(\mathrm{k}\QQ,V) := \!\!\!\!\bigoplus_{\alpha:i\rightarrow j\in \QQ_1}\Hom_\mathrm{k}(V_i,V_j).
\]
The vector  $\mathbf{d}$ is called the \emph{dimension vector} of these representations. We denote by $\GL^\bullet(V):=\prod_{i\in \QQ_0} \textrm{GL}(V_i)$ be the group of graded automorphisms of $V$, then $\GL^\bullet(V)$ acts on $R(\A,V)$ by change of basis: given $g=(g_i)_{i\in\QQ_0}\in \GL^\bullet(V)$ and $M=(M_\alpha)_{\alpha\in \QQ_1}\in R(\A,V)$ the representation $g\cdot M$ is defined by $(g\cdot M)_\alpha=g_{t(\alpha)}\circ M_\alpha\circ g_{s(\alpha)}^{-1}$. Two representations  are isomorphic if they belong to the same $\GL^\bullet(V)$-orbit.

Let $M\in R(\A,V)$, let $B_i$ be a $\mathrm{k}$-basis of $V_i$ for every $i\in \QQ_0$ and let $B$ be the disjoint union of these sets $B_i$. We denote by $\Gamma(M):=\Gamma(M,B)$ the \emph{coefficient quiver} of $M$ with respect to the basis $B$ \cite{Ringel}; this is the quiver whose vertices are the elements of
$B$ and the arrows are defined naturally as follows: for each arrow $\alpha\in\QQ_1$ and every element $b\in B_{s(\alpha)}$ we have \[M_{\alpha}(b)=\sum_{c\in B_{t(\alpha)}}\lambda_{b,c}^{\alpha}c \]
with $\lambda_{b,c}^{\alpha}\in \mathrm{k}$, then for each $\lambda_{b,c}^{\alpha}\neq 0$ we draw an arrow $b\rightarrow c$ with label $\alpha$ in $\Gamma(M)$ \cite{Ringel}. Thus, the coefficient quiver reflects the coefficients corresponding to the representation $M$ with respect to the chosen basis $B$ and will help us to depict representations in a nice way in the remainder of the article. In case there are no multiple arrows between two vertices, we label the arrows of the coefficient quiver with the actual value of $\lambda_{b,c}^{\alpha}$.

We include a basic example in order to display the ideas behind our setup. We will come back to this example throughout this section.
\begin{example}\label{Ex:Jordan}
Let $\QQ$ be the one-loop quiver, that is, $\QQ_0=\{x\}$ and $\QQ_1=\{\alpha:x\rightarrow x\}$, let $V=\mathrm{k}^n$, and consider the admissible ideal $I=(\alpha^n)\subseteq \mathrm{k}\QQ$. Then $R(\A,V)=\mathcal{N}=\{N\in\mathrm{k}^{n\times n}\mid N^n=0\}$ equals the nilpotent cone and $\GL^\bullet(V)=\GL(\mathrm{k}^n)$. Thus, the $\GL^\bullet(V)$-action on $R(\A,V)$ is the usual conjugation action, its orbits are described by the Jordan canonical form \cite{Jor}, or by partitions, that is, combinatorial objects named Young diagrams. The closure relations are known by Gerstenhaber \cite{Ger} and are given by box dropping of Young diagrams.
\end{example}
\subsection{Symmetric quiver representations}\label{Subsec:SymmRep}

Let $\A=\mathrm{k}\QQ/I$ be a symmetric quiver algebra with respect to an anti-involution $\sigma$. 
The anti-involution $\sigma$ can be extended to an isomorphism $\sigma:\A\rightarrow \A^{op}$ of $\A$ to its opposite algebra. This isomorphism induces an equivalence $\sigma:\Rep(\A)\rightarrow \Rep(\A^{op})$ of the representation categories;  by composing with the standard $\mathrm{k}$-duality $D=\Hom(-,\mathrm{k})$ we get a self-duality $\{\}^\ast:\Rep(\A)\rightarrow \Rep(\A)$ on $\Rep(\A)$. With abuse of notation, for a vector space $V$ we denote by 
$V^\ast=\Hom(V,\mathrm{k})$ its linear dual and for a linear map $f:U\rightarrow V$ we denote by $f^\ast:V^\ast \rightarrow U^\ast$ its linear dual  defined by $f^\ast(h)(u)=h(f(u))$ for every $h\in V^\ast$ and $u\in U$. 
For a $\QQ_0$-graded vector space $V=\oplus_{i\in \QQ_0} V_i$,  its \emph{twisted dual} $\nabla V=V^\ast$ is the $\QQ_0$-graded vector space whose $i$-th component is $(\nabla V)_i=(V_{\sigma(i)})^\ast$.  Thus, given a representation $M=(M_\alpha)\in R(\A,V)$ its dual is the representation $M^\ast\in R(\A,\nabla V)$ given by $(M^\ast)_\alpha=M^\ast_{\sigma(\alpha)}$. For our purposes it is convenient to slightly modify this self-duality as follows.
\begin{definition}
Let $\nabla:\Rep(\A)\rightarrow \Rep(\A)$ be the functor $\nabla=-\{\}^\ast$. Thus
\begin{itemize}
\item $\nabla M=-M^\ast$ for every object $M\in \Rep(\A)$;
\item $\nabla h=h^\ast$ for every morphism $h$.
\end{itemize}
\end{definition}
 
Let us fix $\ee\in\{\pm 1 \}$ and let $\langle-,-\rangle:V\times V\rightarrow \mathrm{k}$ be a non-degenerate bilinear form which fulfills two conditions: 
\begin{enumerate}
\item the form $\form$ is \emph{compatible with $\sigma$}, i.e. $\langle-,-\rangle|_{V_i\times V_j}=0$ if $j\neq \sigma(i)$; 
\item the form $\form$ is an \emph{$\ee$-form}: i.e. $\langle v,w\rangle=\ee\langle w,v\rangle$ for every $v,w\in V$.
\end{enumerate}
The pair $(V,\form)$ is called an $\ee$-quadratic space for $(\A,\sigma)$. We highlight some obvious properties of $(V,\form)$:
\begin{itemize}
\item The dimension vector of $V$ is \emph{$\sigma$-symmetric}, i.e. $\mathbf{d}_{\sigma(i)}=\mathbf{d}_i$ for every $i\in\QQ_0$.  
\item There is an isomorphism of $\QQ_0$-graded vector spaces $\Psi:V\rightarrow \nabla V$ given by $v\mapsto \langle v,-\rangle$. We freely identify $V$ and $\nabla V$ by $\Psi$.
\item Every endomorphism $f$ of $V$ has a unique \emph{adjoint} $f^\star$ with respect to $\form$ defined by the condition $\langle v,f(w)\rangle=\langle f^\star(v),w\rangle$, for all $v,w,\in V$. 
\item Every representation $M\in R(\A,V)$ can be naturally seen as an endomorphism of $V$ and we denote by $M^\star$ its adjoint.
\end{itemize}

We denote by $\G(V,\form)=\{g\in \GL(V)| g=(g^{\star})^{-1}\}$ the group of isometries of $(V,\langle-,-\rangle)$. Thus $\G(V,\form)=\OO_d$ is the orthogonal group if $\ee=1$ and it is the symplectic group $\SP_d$  if $\ee=-1$,
where $d$ is the dimension of $V$.

Following \cite{DW} we say that $M\in R(\A,V)$ is an \emph{$\ee$-representation} of $(\A,\sigma)$ with respect to $(V,\form)$ if 
\begin{enumerate}[resume]
\item $M^\star+M=0$.
\end{enumerate}
In other words, $M$ is an $\ee$-representation if, interpreted as an endomorphism of $V$, it lies in the Lie algebra of  $\G(V,\form)$. A $+1$-representation is called \emph{orthogonal} and a $-1$-representation is called \emph{symplectic}.  By identifying $V$ and $\nabla V$ via $\Psi$, the equation $M^\star+M=0$ is rewritten as $\nabla M=M$. We collect all $\ee$-representations in a variety $R(\A,V)^{\form, \ee}=\{M\in R(\A,V)|\, \nabla M=M\}$ and denote by   
$\G^\bullet(V,\form):=\G(V,\form)\cap \GL^\bullet(V)$ the group of graded isometries of $(V,\form)$.  Then the action of $\GL^\bullet(V)$ on $R(\A,V)$  induces an action of $\G^\bullet(V,\form)$ on $R(\A,V)^{\form, \ee}$ by change of basis (\cite{DW,BCI}). One first question which suggests itself is whether or not 
\[\G^\bullet(V,\form)\cdot M = \GL^\bullet(V)\cdot M \cap R(\A,V)^{\form, \ee}\] 
holds true for every $M\in R(\A,V)^{\form, \ee}$, that is, whether the orbits of the smaller group are induced by the orbits of the bigger group. This question is answered positively by Derksen and  Weyman in \cite{DW} and with different techniques in \cite{BCI}. The main question which we address in this article follows immediately:

\begin{question}\label{Question}
 Is it true that 
\[\overline{\G^\bullet(V,\form)\cdot M} = \overline{\GL^\bullet(V)\cdot M} \cap R(\A,V)^{\form, \ee}\] for every $M\in R(\A,V)^{\form, \ee}$?

\end{question}

Main Question \ref{Question} is answered positively in \cite{BCI} for Dynkin quivers. Its answer is particularly interesting when the algebra $\A$ is \emph{representation-finite}, that is, in case there is only a finite number of $\GL^\bullet(V)$-orbits in $R(\A,V)$, as in Example \ref{Ex:Jordan}. Our aim in this article is to give a counterexample of a representation-finite algebra for which the answer to Main Question \ref{Question} is negative. This is indeed unexpected, since our example is closely related to the fundamental Example \ref{Ex:Jordan1} which does not give a counterexample.

\begin{example}\label{Ex:Jordan1}
In case of Example \ref{Ex:Jordan}, we fix $\ee$ to be $+1$ or $-1$ (note that $n$ is supposed to be even in the latter case). Let $J_k$ be the $k\times k$-anti-diagonal matrix with every entry on the anti-diagonal being one and every other entry being zero. The non-degenerate bilinear form $\form:V\times V\rightarrow \mathrm{k}$ given by the matrix $J_n$ if $\ee=1$ and by 
\[F=\left[\begin{array}{cc}0&J_l\\- J_l&0\end{array}\right]\]
if $\ee=-1$ fulfills conditions (1) and (2). Then $\G^\bullet(V,\form)=\OO_n$ if  $\ee=1$ and $\G^\bullet(V,\form)=\SP_n$ if $\ee=-1$ and the $\G^\bullet(V,\form)$-action on $R(\A,V)^{\form, \ee}=\mathcal{N}\cap \Lie G$ (here $\mathcal{N}$ denotes the nilpotent cone as in Example~\ref{Ex:Jordan}) is given by orthogonal/symplectic conjugation. The orbits of the latter are classified by Springer and Steinberg by so-called $\ee$-partitions and their closures are known by Hesselink (these results are e.g. described by Kraft and Procesi in \cite{KP}).  Main Question \ref{Question} is answered positively. 
\end{example}

\subsection{Motivation}
Example \ref{Ex:Jordan} shows that - in addition to being interesting from a quiver representation-theoretic point of view - the answer of Main Question \ref{Question} has further applications to algebraic Lie Theory. This will be worked out in more detail in Section~\ref{Rem:LieTh}.

\begin{remark}\label{rem:MWZ}
Our setup fits into a more general context described by Magyar, Weyman and Zelevinsky in \cite{MWZ}. In fact,  given a complex algebraic variety $X$ together with an action of a group $\G$ and two involutions $\rho:\G\rightarrow \G$ and $\Delta: X\rightarrow X$ such that $\Delta(g\cdot \Delta x)=g^\rho\cdot x$, we denote the fixed point sets by $\G^\rho\subset \G$ and $X^\Delta\subset X$. Assume that
\begin{enumerate}[label=(\arabic*)]
\item the group $\G$ is a subgroup of the group of invertible elements $E^\times$ of a finite-dimensional associative algebra $E$ over $\mathrm{k}$;
\item the anti-involution of $\G$ given by $g\mapsto g^\ast:=(g^\rho)^{-1}$ extends to a $\mathrm{k}$--linear anti-involution $f\mapsto f^\ast$ on the algebra $E$;
\item for every fixed point $x\in X^\Delta$, its stabilizer $H=\textrm{Stab}_{\G}(x)$ is the group of invertible elements of its linear span $\textrm{Span}_\mathrm{k}(H)\subset E$.
\end{enumerate}
Then $\G x\cap X^\Delta=\G^\rho x$ holds true for all $x\in X^\Delta$ by \cite[Section~2.1]{MWZ}.
 The natural subsequent (and open) question  is
\begin{equation}\label{Eq:MWZ}
 "\mathrm{Is~ it~ true ~that~} \overline{\G x}\cap X^\Delta=\overline{\G^\rho x} \textrm{~for~ every~} x\in X^\Delta\textrm{?}"
\end{equation}
As described in \cite[Subsection 2.4]{BCI}, $R(\A,V)$ and $R(\A,V)^{\form, \ee}$ can be realized as $X$ and $X^\Delta$; and $\GL^\bullet(V)$ and  $\G^\bullet(V,\form)$ can be realized as $G$ and $G^\rho$. Thus, our counterexample on Main Question \ref{Question} is also  a counterexample for  (\ref{Eq:MWZ}).
\end{remark}

\section{Ext-, deg- and hom-order}\label{Sec:orders}

Let $\A$ be a quiver algebra, let $\mathbf{d}\in \mathbf{Z}_{\geq0}^{\QQ_0}$ be a dimension vector and let $V$ be a $\QQ_0$-graded complex vector space of dimension vector $\mathbf{d}$.

Let  $M,N\in R(\A,V)$. We denote  $[M,N]:=\textrm{dim Hom}_{\A}(M,N)$ and $[M,N]^1:=\textrm{dim Ext}^1_{\A}(M,N)$ and define three partial orders on  $R(\A,V)$ which were first described by Abeasis-Del Fra for quivers of Dynkin type $A$ \cite{AbeasisDelFraEquioriented, AbeasisDelFra}, before being generalized to quiver algebras by Riedtmann \cite{Riedtmann}, Bongartz \cite{Bongartz} and Zwara \cite{Zwara}.

\begin{itemize}
\item
The \emph{degeneration order} $\degg$ is defined by 
\[
M\degg N :\Longleftrightarrow N\in \overline{\GL^\bullet(V)\cdot M}
\]
\item The \emph{Hom-order} $\leq_{\Hom}$  is defined by 
\[
M\leq_{\Hom} N  :\Longleftrightarrow  [M,E]\leq [N,E] \textrm{ for every indecomposable }E.
\]
\item 
The \emph{Ext-order} $\leq_{\Ext}$ is defined by
\[\begin{array}{ccc}
&& \exists M(1),\cdots, M(k)\in R(\A,V) \textrm{ and short exact }\\
M\leq_{\Ext} N &:\Longleftrightarrow&\textrm{ sequences }0\rightarrow U(i)\rightarrow M(i-1)\rightarrow V(i)\rightarrow 0~~~ (\forall i)\\
& &\textrm{ such that }M(1)=M, M(k)=N, \, M(i)\simeq U(i)\oplus V(i).
\end{array}\]
\end{itemize} 

It is known  by \cite[Lemma~1.1]{Bongartz} (first implication) and \cite[Proposition~2.1]{Riedtmann} (second implication) that
\[
\xymatrix{
M\leq_{\Ext}N\ar@{=>}[r]&M\degg N\ar@{=>}[r]&M\leq_{\Hom}N.
}
\]

If $\A$ is a representation-finite algebra, then Zwara \cite[Corollary of Theorem 1]{Zwara} shows
\[\xymatrix{
M\degg N\ar@{<=>}[r]&M\leq_{\Hom}N.
} \]
If furthermore all indecomposables are rigid, i.e. $[E,E]^1=0$ for all indecomposables $E$, then all three orders coincide \cite[Theorem~2]{Zwara}. In particular, they are equivalent for Dynkin quivers. Note that the result on Dynkin quivers also follows from work of Bongartz; he shows that all three partial orders coincide for representation-directed algebras \cite[Proposition 3.2,Corollary 4.2]{Bongartz}.

Following \cite{BCI}, we introduce symmetric versions of $\degg$ and $\leq_{\Ext}$ now.  Thus, we assume $\A$ to be a symmetric quiver algebra, let $\ee$ be $+1$ or $-1$ and let $\form$ be a bilinear form as in the Section \ref{Sec:Setup}.  Then we consider the following partial orders on $R(\A,V)^{\form, \ee}$: let $M,N\in R(\A,V)^{\form, \ee}$.
\begin{itemize}
\item
The \emph{symmetric degeneration order} $\degg^\ee$  is defined by 
\[
M\degg^\ee N :\Longleftrightarrow N\in \overline{\GL^\bullet(V,\form) \cdot M.
}
\]
\item 
The \emph{symmetric Ext-order} $\leq^\ee_{\Ext}$ is defined by
\[\begin{array}{ccc}
&& \exists M(1),\cdots, M(k) \in R(\A,V)^{\form, \ee} \textrm{ and s.e.s.}\\M\leq^\ee_{\Ext} N &:\Longleftrightarrow&
 0\rightarrow U(i)\rightarrow M(i-1)\rightarrow V(i)\rightarrow 0~~~ (\forall i)~\textrm{such that }\\
&&M(1)=M, M(k)=N, U(i)\textrm{ is isotropic in }M(i-1),\\&&\textrm{ and } M(i)\simeq U(i)\oplus \nabla U(i)\oplus U(i)^{\perp}/U(i).
\end{array}\]
\end{itemize}
(Here $U(i)^\perp$ denotes the orthogonal subspace of $U(i)$ in $M(i)$.) 
It is known  by \cite[Corollary~3.3]{BCI} and by the fact that $\GL^\bullet(V,\form)\subseteq \GL^\bullet(V)$ is a subgroup that
\[
\xymatrix{
M\leq^\ee_{\Ext}N\ar@{=>}[r]&M\degg^\ee N\ar@{=>}[r]&M\degg N~(\ar@{=>}[r]&M\leq_{\Hom}N).
}
\]

\begin{question}\label{Question2}
Does $
\xymatrix{
\degg^\ee \ar@{<=>}[r]&\degg 
}
$
hold true on $R(\A,V)^{\form, \ee}$?
\end{question}

From our considerations before, it is clear that Main Question \ref{Question} and Main Question \ref{Question2} coincide; we can thus answer either of them.
\section{The Seesaw algebra}\label{Sec:Seesaw}

Let $n\in\{2l,2l+1\}$ be an integer and let  $\A=\A_n=\mathrm{k}\QQ/I$ be the symmetric quiver algebra given by the symmetric quiver
\[
\QQ:\;\xymatrix{
1\ar^{a_1}[r]&2\ar^{a_2}[r]&\cdots\ar^{a_{l-1}}[r]&l\ar^{a_l}[r]&\omega\ar_{\gamma=\gamma^\ast}@(ld,rd)\ar^{a_l^\ast}[r]&l^\ast\ar^{a_{l-1}^\ast}[r]&\cdots\ar^{a_2^\ast}[r]&2^\ast\ar^{a_1^\ast}[r]&1^\ast
}
\]
where $\sigma(i)=i^\ast$ for $i\in\QQ_0\cup\QQ_1$ and by the admissible ideal $I=(\gamma^2,a_l^\ast a_l)$. We call it the \emph{Seesaw algebra}.  We consider the symmetric dimension vector  
\begin{equation}\label{Eq:DefDimD}
\mathbf{d}:=(\mathbf{d}_i)_i= (1,2,\cdots,l-1,l,n,l,l-1,\cdots, 2,1)
\end{equation}

and a $\QQ_0$-graded vector space $V$ of graded dimension $\mathbf{d}$. Let $\form$ be a bilinear $\ee$-form on $V$ as in Section \ref{Sec:Setup}. In order to be able to work in coordinates (and to depict our representations nicely), let us fix a basis $\mathcal{B}_s=\{v_k^{(s)}\mid 1\leq k\leq i\}$ of each $V_s$ where $s\in \{i, i^\ast\}$ and  $\mathcal{B}_\omega:=\{v_k^{(\omega)}, v_{k}^{(\omega^\ast)}\mid 1\leq k\leq l\}$ or $\mathcal{B}_\omega:=\{v_k^{(\omega)}, v_{k}^{(\omega^\ast)}\mid 1\leq k\leq l\}\cup \{v\}$ of $V_{\omega}$, such that on the basis elements the form is zero unless
\[\langle v^{(i)}_k, v^{(i^\ast)}_{k} \rangle=1,~\langle v^{(i^\ast)}_k, v^{(i)}_{k} \rangle=\ee,~\langle v^{(\omega)}_k, v^{(\omega^\ast)}_{k} \rangle=1~\mathrm{or}~\langle v^{(\omega^\ast)}_{k}, v^{(\omega)}_{k} \rangle=\ee.\]

Let $M=(M_{\beta})_{\beta\in\QQ_1}$ and $N=(N_{\beta})_{\beta\in\QQ_1}$ be two representations in $R(\A,V)$.  Here, 
$M_{a_i}=N_{a_i}$,  are the standard embeddings into the first $i$ copies of $\mathrm{k}$, $M_{a_{i^\ast}}=N_{a_{i^\ast}} $ equals minus the standard projection of the last $i$ copies of $\mathrm{k}$ onto $\mathrm{k}^{i}$. Furthermore, $M_{\gamma}$  sends $v_1^{(\omega)}$ to $v_l^{(\omega^\ast)}$,  $v_l^{(\omega)}$ to $-\ee v_1^{(\omega^\ast)}$ and $N_{\gamma}$  sends $v_1^{(\omega)}$ to $v_l^{(\omega)}$,  $v_l^{(\omega^\ast)}$ to $-v_1^{(\omega^\ast)}$ and every other basis element is mapped to zero by $M_{\gamma}$ and $N_{\gamma}$.
We  depict them by their coefficient quivers for $n=2l=4$. 
\[\xymatrixrowsep{0.01pc}\xymatrixcolsep{1.4pc}
\xymatrix{
&v^{(1)}_1\ar^1[r]&v^{(2)}_1\ar^1[r]&v^{(\omega)}_1\ar@/^1.7pc/^1[dd]&&\\
\Gamma(M)=&&v^{(2)}_2\ar^(.36)1[r]&v^{(\omega)}_2\ar@/_2.1pc/_{-\ee}[dd]&&\\
&&&v^{(\omega^\ast)}_{2}\ar_(.54){-1}[r]&v^{(2^\ast)}_2&&\\
&&&v^{(\omega^\ast)}_{1}\ar_{-1}[r]&v^{(2^\ast)}_1\ar_{-1}[r]&v^{(1^\ast)}_1\\}
\]

\[\xymatrixrowsep{0.01pc}\xymatrixcolsep{1.4pc}
\xymatrix{
&v^{(1)}_1\ar^1[r]&v^{(2)}_1\ar^1[r]&v^{(\omega)}_1\ar@/^1.7pc/^1[d]&&\\
\Gamma(N)=&&v^{(2)}_2\ar^1[r]&v^{(\omega)}_2&&\\
&&&v^{(\omega^\ast)}_{2}\ar_{-1}[r]\ar@/_1.7pc/_{-1}[d]&v^{(2^\ast)}_2&\\
&&&v^{(\omega^\ast)}_{1}\ar_{-1}[r]&v^{(2^\ast)}_1\ar_{-1}[r]&v^{(1^\ast)}_1\\
}\]
Then  the relations  $M_{\gamma}^2=N_{\gamma}^2=0=\pi_l\circ\iota_l$  are fulfilled and  $M, N\in R(\A,V)^{\form, \ee}$. In order to approach our counterexample, we compute the stabilizer and orbit dimensions in the following lemma. For simplicity of notation, we put $\G:=\GL^\bullet(V)$, $R:=R(\A,V)$, $\G(\ee):=\G^\bullet(V,\form)$ and $R(\ee)=R(\A,V)^{\form, \ee}$. 

\begin{lemma}\label{Lem:Stabilizers}
\[\begin{tabular}{|c|c|c|c|}
\hline
&$\dim G $&$\dim\mathrm{Stab}_G(M)$& $\dim\mathrm{Stab}_G(N)$ \\\hline
Type A&$2(\sum_{i=1}^l i^2)+n^2$ &(2l-1)(l-2)+3 &(2l-1)(l-2)+4 \\\hline\hline
&$\dim G(\ee) $&$\dim\mathrm{Stab}_{G(\ee)}(M)$& $\dim\mathrm{Stab}_{G(\ee)}(N)$\\\hline
Type B&$\frac{1}{2}(\dim G-n) $&l(l-2)+2&l(l-2)+3 \\
Type C&$\frac{1}{2}(\dim G+n)$&l(l-2)+1 &l(l-2)+2\\ 
Type D&$\frac{1}{2}(\dim G-n)$&(l-1)(l-2)+2 &(l-1)(l-2)+2\\\hline
\end{tabular}\]
\end{lemma}

\begin{proof}
The stabilizer dimension can e.g. be calculated by basic methods of linear algebra when going over to Borel-orbits of $2$-nilpotent matrices as explained in Section~\ref{Rem:LieTh}. Another option is a calculation of their endomorphism spaces \cite{BCIE} or of certain Crawley Boevey triples \cite{CB} (since $\A$ is a string algebra). 
\end{proof}
Note that orbit dimensions can then be read off, since $\dim G.X = \dim G - \dim \mathrm{Stab}_G(X)$ and  $\dim G(\ee).X = \dim G(\ee) - \dim \mathrm{Stab}_{G(\ee)}(X)$ for $X=M,N$.

Let us fix $n=2l$ for now, let $V=\oplus_{i\in\QQ_0}V_i=\oplus_{i\in\QQ_0}\mathrm{k}^{\mathbf{d}_i}$  and fix $\ee = 1$ (that is, we work in orthogonal type D). The following proposition gives the claimed counterexample. 
\begin{proposition}\label{Prop:Counterex} 
For $n=2l$, i.e. in type $D$, $M$ and $N$ gives a negative answer to Main Question~\ref{Question2}, which means
\begin{enumerate}
\item $N\in \overline{\G M}$, i.e. $M\degg N$.
\item $N\notin\overline{\G(1)\cdot M}$, i.e.  $M\nleq_{\deg}^\ee N$.
\end{enumerate}
\end{proposition}
\begin{proof}
The representation $M$ corresponds to the so-called oriented link pattern (i.e. an oriented graph defined in \cite{BoRe} representing the part of the coefficient quiver which describes the loop $\gamma$ at vertex $\omega$)
\[\xygraph{ !{<0cm,0cm>;<0.7cm,0cm>:<0cm,1cm>::}
!{(-3.5,0) }*+{\underset{1}\bullet}="1"
!{(-2.5,0) }*+{\underset{2}\bullet}="2"
!{(-1.5,0)}*+{\cdots}="3"
!{(-0.5,0)}*+{\underset{l}\bullet}="4"
!{(0.5,0) }*+{\underset{l^\ast}\bullet}="5"
!{(1.5,0)}*+{\cdots}="6"
!{(2.5,0)}*+{\underset{2^\ast}\bullet}="7"
!{(3.5,0)}*+{\underset{1^\ast}\bullet}="8"
"1":@/^0.75cm/"5"
"4":@/^0.75cm/"8"
}\]
and the representation $N$ corresponds to
\[\xygraph{ !{<0cm,0cm>;<0.7cm,0cm>:<0cm,1cm>::}
!{(-3.5,0) }*+{\underset{1}\bullet}="1"
!{(-2.5,0) }*+{\underset{2}\bullet}="2"
!{(-1.5,0)}*+{\cdots}="3"
!{(-0.5,0)}*+{\underset{l}\bullet}="4"
!{(0.5,0) }*+{\underset{l^\ast}\bullet}="5"
!{(1.5,0)}*+{\cdots}="6"
!{(2.5,0)}*+{\underset{2^\ast}\bullet}="7"
!{(3.5,0)}*+{\underset{1^\ast}\bullet}="8"
"1":@/^0.5cm/"4"
"5":@/^0.5cm/"8"
}.\]

As shown in \cite[Theorem~4.6]{BoRe}, there is a minimal degeneration from $M$ to $N$.
Note that, even though the quiver examined in \cite{BoRe} is different from $\mathcal{Q}$, the fact that in both setups degenerations correspond to Borel-orbit closure relations of $2$-nilpotent matrices (see \cite[Lemma 3.2]{BoRe}  and Section~\ref{Rem:LieTh}) makes sure that the description of the degeneration order given in \cite{BoRe} is valid.

By part (3) of Lemma~\ref{Lem:Stabilizers} the symmetric orbits $\G(1)\cdot M$ and $\G(1)\cdot N$ have the same codimension in $R(1)$. This implies part (2). 
\end{proof}

Proposition \ref{Prop:Counterex} leads to negative answers for Main Question \ref{Question} and Main Question \ref{Question2} for the Seesaw algebra and thus we have the following corollary.
\begin{corollary}
Given a symmetric quiver algebra of finite representation type, the equivalence
$\degg \Longleftrightarrow \degg^\ee$ is not in general true in $R(\A,V)^{\form, \ee}$. 
\end{corollary}

\section{Connection with Borel orbits}\label{Rem:LieTh}
The (symmetric) representation theory of the Seesaw algebra can be translated to a particular Lie-theoretic setup  in a related, but more involved way as in Example \ref{Ex:Jordan1}. In this section we recall this connection worked out in \cite{BCIE}.
\subsection{Borel orbits induced by type $A$}
Let $n=2l$ or $n=2l+1$ be a positive integer and let us consider the complex vector space $U=\mathrm{k}^n$.  Let $\ee$ be $1$ or $-1$ and let $\langle-,-\rangle_U$ be a non-degenerate and bilinear $\ee$-form on $U$. Given a linear endomorphism $f$ of $U$ we denote by $f^t$ its adjoint with respect to this form. The general linear group $\GL(\mathrm{k}^n)$ is endowed with the involution $\rho$ given by $g^\rho=(g^t)^{-1}$ and we denote by $\G(\ee)=\GL(\mathrm{k}^n)^\rho$ the set of fixed points. Thus $\G(\ee)$ is the symmetry group of $(U,\langle-,-\rangle_U)$ and it hence coincides with the orthogonal group $O_n$ if $\ee=1$ and with the symplectic group $SP_n$ if $\ee=-1$.  

We denote by  $\mathcal{N}^{(2)}$ the subset of $\Lie(\GL(\mathrm{k}^n))$ of $2$-nilpotent $n\times n$ matrices. We denote by $\Delta:\mathcal{N}^{(2)}\rightarrow \mathcal{N}^{(2)}$ the involution given by $\Delta(A)=-A^t$ and denote by $\mathcal{N}^{(2)}(\ee)=(\mathcal{N}^{(2)})^{\Delta}$ the set of fixed points. We then see that $\mathcal{N}^{(2)}(\ee)\subset \Lie(\G(\ee))$. 
 
Let us fix a maximal isotropic flag $U_1\subset \cdots\subset U_{l}\subset U$ of $U$, i.e. $\textrm{dim }U_i=i$ for all $i$ and $U_l$ is a totally isotropic subspace of maximal dimension. This gives rise to a complete flag $\mathcal{F}_\bullet= U_1\subset \cdots\subset U_{l}\subseteq U_{l}^\perp\subset U_{l-1}^\perp\subset\cdots U_1^{\perp}\subset U$. Let $B\subset GL(\mathrm{k}^n)$ be the stabilizer of $\mathcal{F}_\bullet$ in $\GL(\mathrm{k}^n)$. Given $b\in B$, $v\in U_i$ and $w\in U_i^\perp=:U_{n-i}$ we have
\[
\langle b^\rho(v), w\rangle_U=\langle v, b^{-1}w\rangle_U=0
\]
and hence we see that $b^\rho$ stabilizes the flag $\mathcal{F}_\bullet$. We conclude that $\rho(B)=B$. We denote by $B(\ee)=B^\rho$ the set of fixed points. It is well-known that $B(\ee)\subset \G(\ee)$ is a Borel subgroup of $\G(\ee)$ since it is the stabilizer in $\G(\ee)$ of a maximal isotropic flag \cite[Section~4.1]{Procesi}.

The general linear group $\GL(\mathrm{k}^n)$ acts on $\mathcal{N}^{(2)}$ by conjugation and thus induces an action of the Borel subgroup $B$. The following compatibility relation holds: for every $A\in \mathcal{N}^{(2)}$ and $b\in B$
\[
\Delta(b\cdot \Delta(A))=b^\rho\cdot A.
\]
It follows that the group of fixed points $B(\ee)$ acts on the set of fixed points $\mathcal{N}^{(2)}(\ee)$. It is straightforward to check that the theorem of Magyar-Weyman-Zelevinsky recalled in Remark~\ref{rem:MWZ} applies to the pair $(B,\mathcal{N}^{(2)})$ (here $E$ is the algebra of the $n\times n$ matrices $x$ with the property that $xU_i\subseteq U_i$ for every $i$) and thus we get  that for every fixed point $A\in \mathcal{N}^{(2)}(\ee)$
\[
B\cdot A\cap \mathcal{N}^{(2)}(\ee)=B(\ee)\cdot A. 
\]
\begin{definition}
We say that the orbit-closure relation of $B(\ee)$ on $\mathcal{N}^{(2)}(\ee)$ is \emph{induced by type $A$} if for every $A\in\mathcal{N}^{(2)}(\ee)$ the following holds
\[
\overline{B\cdot A}\cap \mathcal{N}^{(2)}(\ee)=\overline{B(\ee)\cdot A}. 
\]
\end{definition}
In  \cite{BCIE} it is shown that  the problem of determining if the orbit-closure relation of $B(\ee)$ on $\mathcal{N}^{(2)}(\ee)$ is induced by type $A$ is related to the Main Question~\ref{Question2}.  Before recalling this, 
we notice that in \cite{BoRe} combinatorial invariants associated to oriented  link patterns (namely $p_i$ and $q_{i,j}$) are defined and it is shown that they describe the orbit closure relations in type $A$ completely. In case of a positive answer to Main Question 3.1, we therefore know that these invariants provide a handy way to explicitly describe $B(\ee)$-orbit closure relations.

Let us briefly recall the construction of  \cite{BCIE}, for convenience of the reader. We consider the Seesaw algebra $\mathcal{A}_n$ endowed with the symmetry $\sigma$. The $\ee$-form $\langle-,-\rangle_U$ on $U$ descends to a non-degenerate $\ee$-form on $U_i\oplus U/(U_i^\perp)$ for every $i$. We consider the  vector space $V=U\oplus\bigoplus_{i=1}^l(U_i\oplus U/(U_i^\perp))$ and define a $\QQ_0$-grading on it by putting $V_\omega=U$, $V_i=U_i$ and  $V_{i^\ast}=U/(U_i^\perp)$, for $i=1,\cdots l$. We see that the $\ee$-form on $U$ induces an $\ee$-form $\langle-,-\rangle$ in $V$  and the pair $(V,\langle-,-\rangle)$ is an $\ee$-quadratic space for $(\A_n,\sigma)$ (see Section~\ref{Subsec:SymmRep} for the definition). Let us consider the representation variety $R(\A,V)$ and its subvariety $R(\A,V)^{\form, \ee}$ of symmetric representations. Let $\mathcal{B}=\A_n/(\gamma)$ and let $M^0\in R(\mathcal{B},V)$ be the representation given as follows: 
\[
\xymatrix@C=13pt{
U_1\ar@{^{(}->}^{j_1}[r]& U_2\ar@{^{(}->}^{j_2}[r]&\cdots  \ar@{^{(}->}^(.56){j_{l-1}}[r]&U_l\ar@{^{(}->}^(.6){j_{l}}[r]&U\ar@{->>}^(.4){-p_l}[r]&V/(U_l^\perp)\ar@{->>}^(.6){-p_{l-1}}[r]&\cdots \ar@{->>}^(.4){-p_{2}}[r]&\ar@{->>}^(.5){-p_{1}}[r]V/(U_2^\perp)&V/(U_1^\perp)
}
\]
where $j_{i}: U_i\rightarrow U_{i+1}$ denotes the inclusion and  $p_i: U/(U_{i+1}^\perp)\rightarrow V/(U_i^\perp)$ the induced surjection for $i=1,\cdots l$. The main property of $M^0$ is that its stabilizer in $\GL^\bullet(V)$ is isomorphic to $B$. Moreover, $M^0$ is also symmetric, i.e. $M^0\in R(\mathcal{B},V)^{\form, \ee}$ and the stabilizer of $M^0$ in the symmetry group $\G^\bullet(V,\form)$ is isomorphic to $B(\ee)$. 
The quotient map of algebras $\A_n\rightarrow \mathcal{B}$ induces the morphisms of affine varieties $\pi: R(\A,V^\bullet)\rightarrow R(\mathcal{B},V^\bullet)$ and $\pi(\ee): R(\A,V^\bullet)^{\form, \ee}\rightarrow R(\mathcal{B},V^\bullet)^{\form, \ee}$ which forget the loop. We notice that $\pi$ is $\GL^\bullet(V)$-equivariant and $\pi(\ee)$ is $\G^\bullet(V,\form)$-equivariant. Let $Y=GL^\bullet(V)\cdot M^0$ denote the orbit of $M^0$ and let $Y(\ee)=\G^\bullet(V,\form)\cdot M^0$ denote the orbit of $M^0$ by the symmetry group. 
Let $X=\pi^{-1}(Y)$ and $X(\ee)=\pi^{-1}(Y(\ee))$. We denote by $p:X\rightarrow Y$ and $p(\ee):X(\ee)\rightarrow Y(\ee)$ the restriction maps. By construction, $p^{-1}(M^0)=\mathcal{N}^{(2)}$ and $p(\ee)^{-1}(M^0)=\mathcal{N}^{(2)}(\ee)$. As shown in \cite[Lemma~3.2]{BoRe} and  \cite[Lemma~4.4]{BCIE}, one can apply  \cite[Theorem~3.1]{BoRe} in this situation (see \cite{Se}),  and get  isomorphisms of complex varieties: 
\[
\begin{array}{ccc}
X\simeq  \GL^\bullet(V)\times^{B}\mathcal{N}^{(2)}&\textrm{and}&X(\ee)\simeq \G^\bullet(V,\form)\times^{B(\ee)}\mathcal{N}^{(2)}(\ee).
\end{array}
\]
Moreover, again by \cite[Theorem~3.1]{BoRe}, the embedding $j:\mathcal{N}^{(2)}\subset X$ and the embedding $j(\ee):\mathcal{N}^{(2)}(\ee)\subset X(\ee)$ send orbits to orbits and orbit closures to orbit closures which means that  for every $A\in \mathcal{N}^{(2)}$ and $A(\ee)\in  \mathcal{N}^{(2)}(\ee)$  
\begin{eqnarray}\label{Eq:J1}
&&j(B\cdot A)=GL^\bullet(V)\cdot j(A);\quad j(\ee)(B(\ee)\cdot A(\ee))=\G^\bullet(V,\form)\cdot j(\ee)(A(\ee));\\\label{Eq:J2}
&&j(\overline{B\cdot A})=\overline{GL^\bullet(V)\cdot j(A)};\quad j(\ee)(\overline{B(\ee)\cdot A(\ee)})=\overline{\G^\bullet(V,\form)\cdot j(\ee)(A(\ee))}.
\end{eqnarray}
In particular, for every $M\in \pi^{-1}(M^0)$ 
\begin{equation}\label{Eq:StabilizersIso}
\textrm{Stab}_G(M)\simeq \textrm{Stab}_B(M_\gamma).
\end{equation}
We notice that the restriction of $j$ to $\mathcal{N}^{(2)}(\ee)\subset \mathcal{N}^{(2)}$ is  $j(\ee)$. Using this compatibility and the fact that $j$ is injective, \eqref{Eq:J1} and \eqref{Eq:J2} imply at once the following equivalence: for every $A(\ee)\in \mathcal{N}^{(2)}(\ee)$ we have 
\[
\xymatrix{
\overline{B(\ee)\cdot A(\ee)}=\overline{B\cdot A(\ee)}\cap \mathcal{N}^{(2)}(\ee)\ar@{<=>}[d]\\
\overline{ \G^\bullet(V,\form)\cdot j(\ee)(A(\ee))}=\overline{\GL^\bullet(V)\cdot j(A(\ee))}\cap j(\mathcal{N}^{(2)}(\ee))
}
\]
which is the claim we wanted to prove. 
As a consequence of the discussion of this section, Proposition~\ref{Prop:Counterex} implies that
\begin{corollary}
Borel orbit closures of type $D$ (meaning $\ee=1$ and $n$ even) are not induced by type $A$.
\end{corollary}
\begin{example}\label{Ex:Seesaw} 
Let $(n,\ee)$ be either $(5,1)$, $(4,-1)$ or $(4,1)$. Figure~\ref{Fig:A} shows the $B$-orbits and their closures inside $\mathcal{N}^{(2)}(\ee)$ and Figure~\ref{Fig:D} shows the $B(\ee)$-orbits and their closures inside $\mathcal{N}^{(2)}(\ee)$.  We see that in type $B$ and $C$, i.e. $(n,\ee)=(5,1)$ and $(n,\ee)=(4,-1)$, Borel orbit closures are induced by type $A$. This can be proved by using the symmetric Ext-ordering and by constructing explicit curves in the symmetric representation varieties. This example was partially worked out with the help of Francesco Esposito and Giovanna Carnovale during a research visit of the second-named author in Padova. 
\begin{figure}
\begin{center}
\begin{tikzpicture}[descr/.style={fill=white,inner sep=2pt}]
\node (A3) at (-3.9,-4) {$\OO_5$};
\node (A1) at (0,-4) {$\SP_4$};
\node (A2) at (3.6,-4) {$\OO_4$};
\node (D3) at (-3.5,0.3)  {\includegraphics[trim=180 150 170 120,clip,width=110pt,height=230pt]{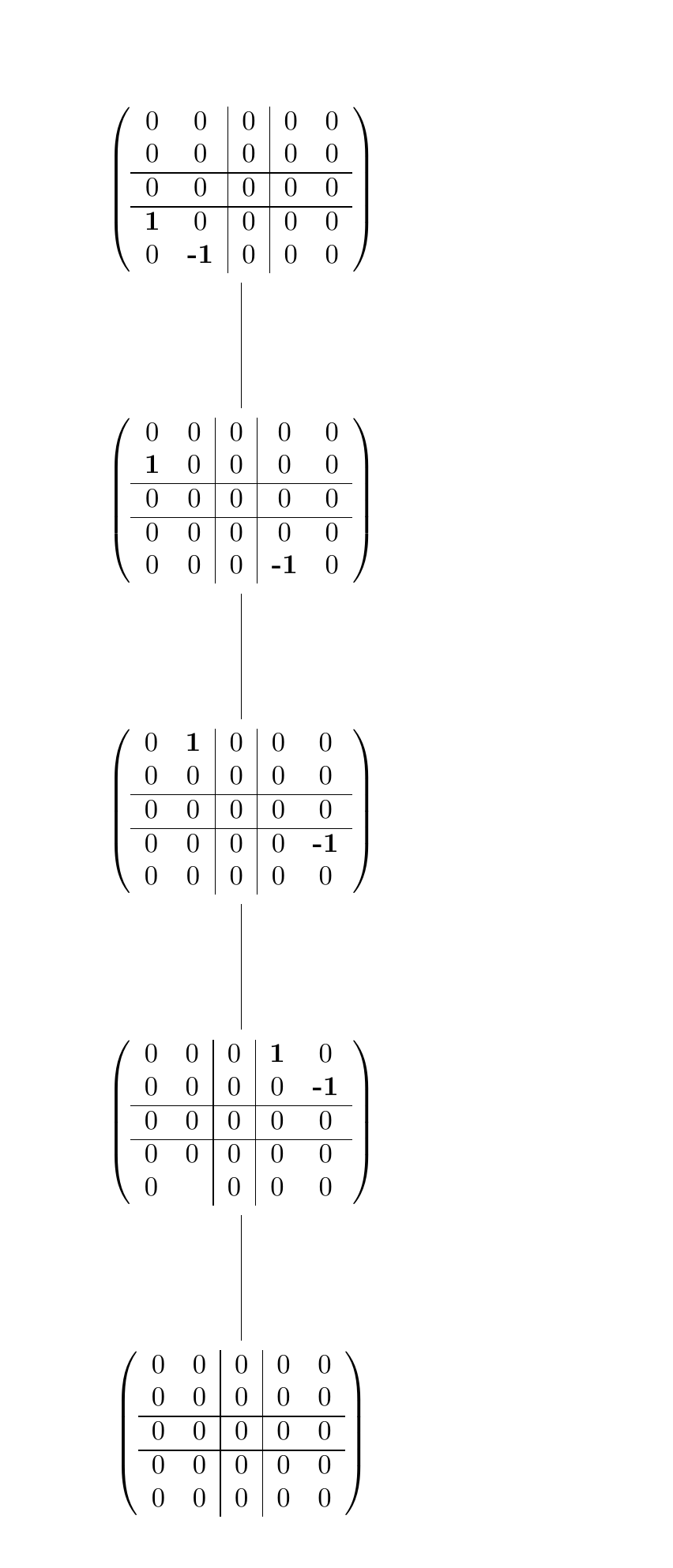}};
\node (D1) at (0.1,0) {\includegraphics[trim=120 210 100 130,clip,width=167pt]{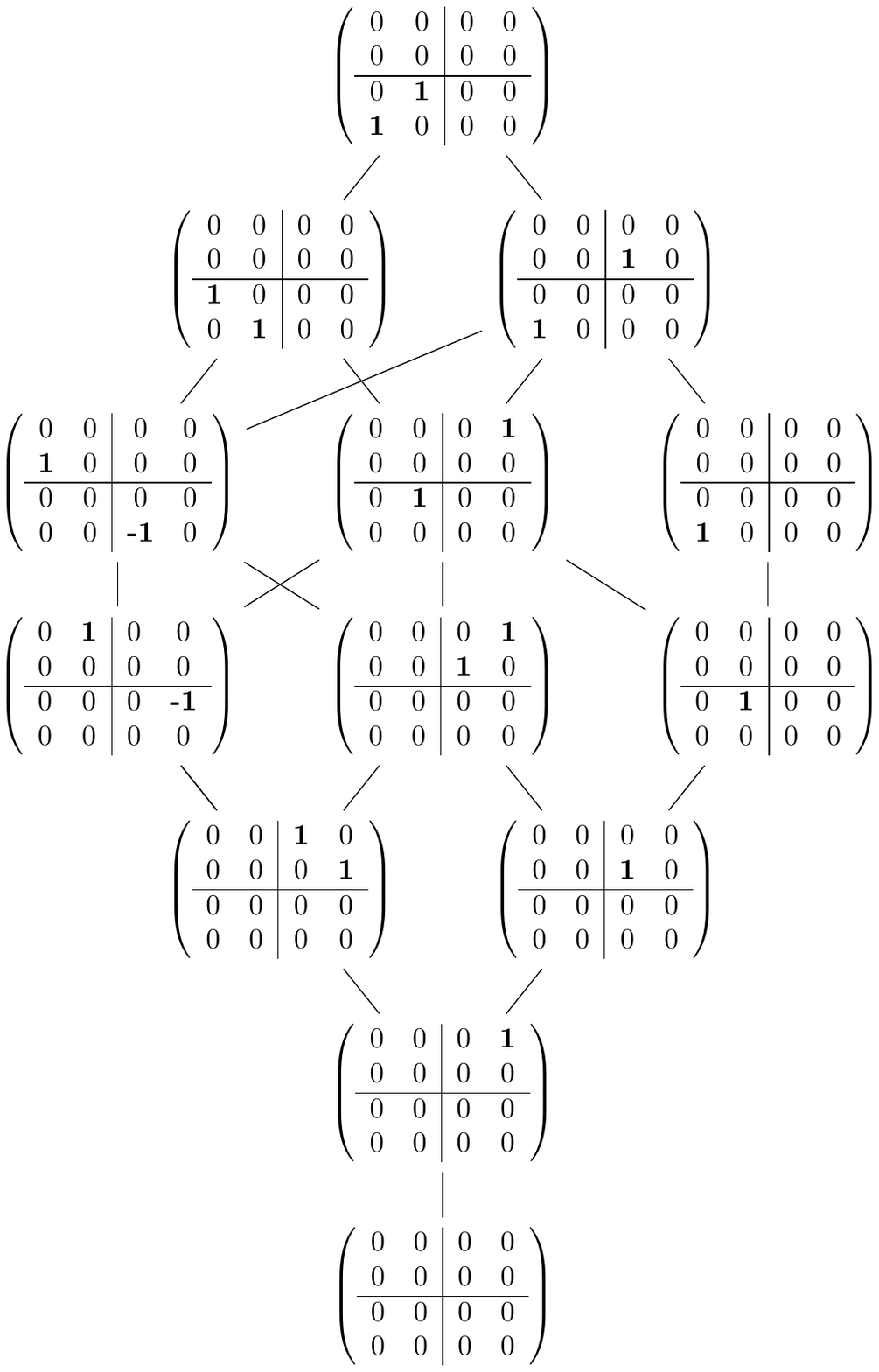}};
\node (D2) at (4,0.3) {\includegraphics[trim=200 225 210 130,clip,width=110pt,height=230pt]{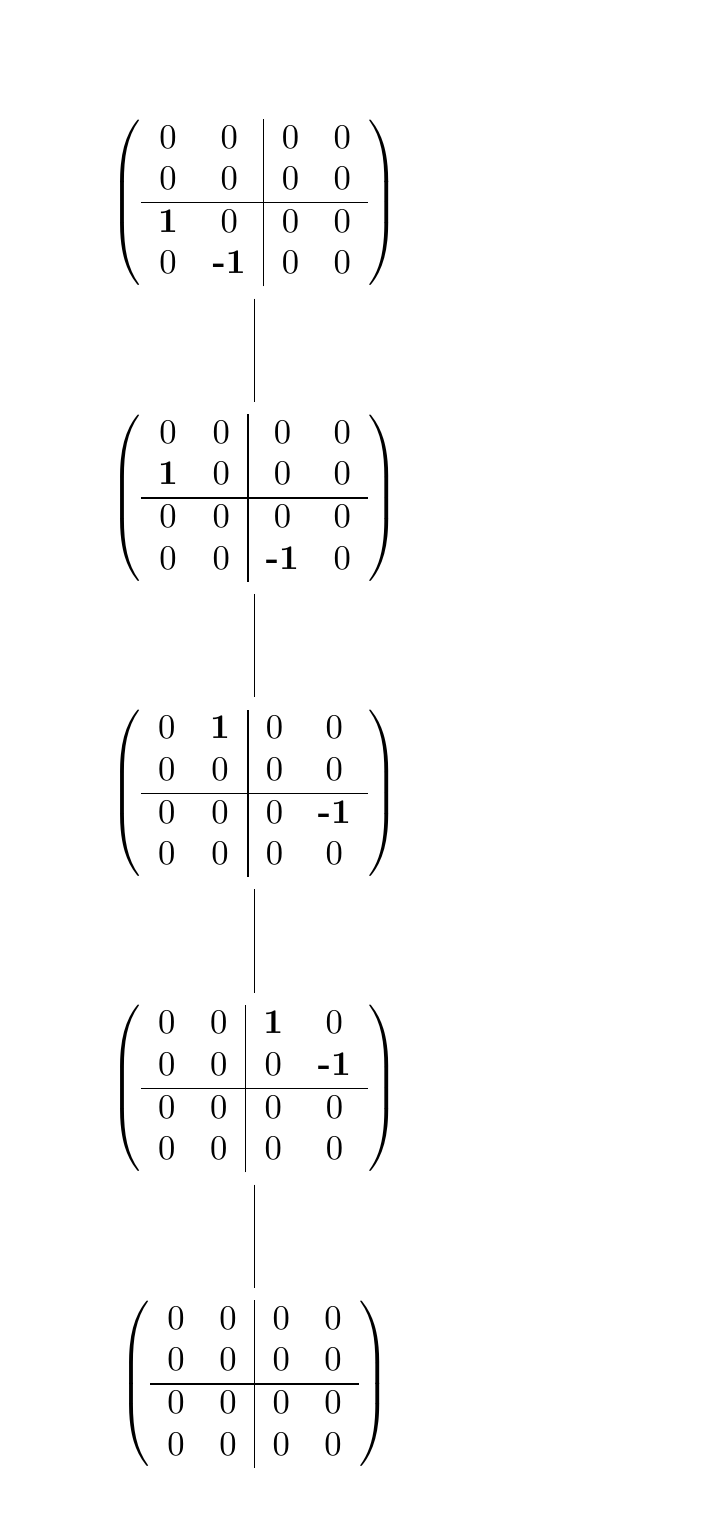}};
  \end{tikzpicture}
\end{center}\caption{$B$-orbits and their closures inside $\mathcal{N}^{(2)}(\ee)$}\label{Fig:A}
\end{figure}

\begin{figure}
\begin{center}
\begin{tikzpicture}[descr/.style={fill=white,inner sep=2pt}]
\node (A3) at (-3.9,-4) {$\OO_5$};
\node (A1) at (0,-4) {$\SP_4$};
\node (A2) at (4.2,-4) {$\OO_4$};
\node (D3) at (-3.5,0.3)  {\includegraphics[trim=180 150 170 120,clip,width=110pt,height=230pt]{SO5_deg.pdf}};
\node (D1) at (0.1,0) {\includegraphics[trim=120 210 100 130,clip,width=167pt]{SP4Matrizen_degs.pdf}};
\node (D2) at (4.2,0.3) {\includegraphics[trim=170 360 205 90,clip,width=90pt,height=180pt]{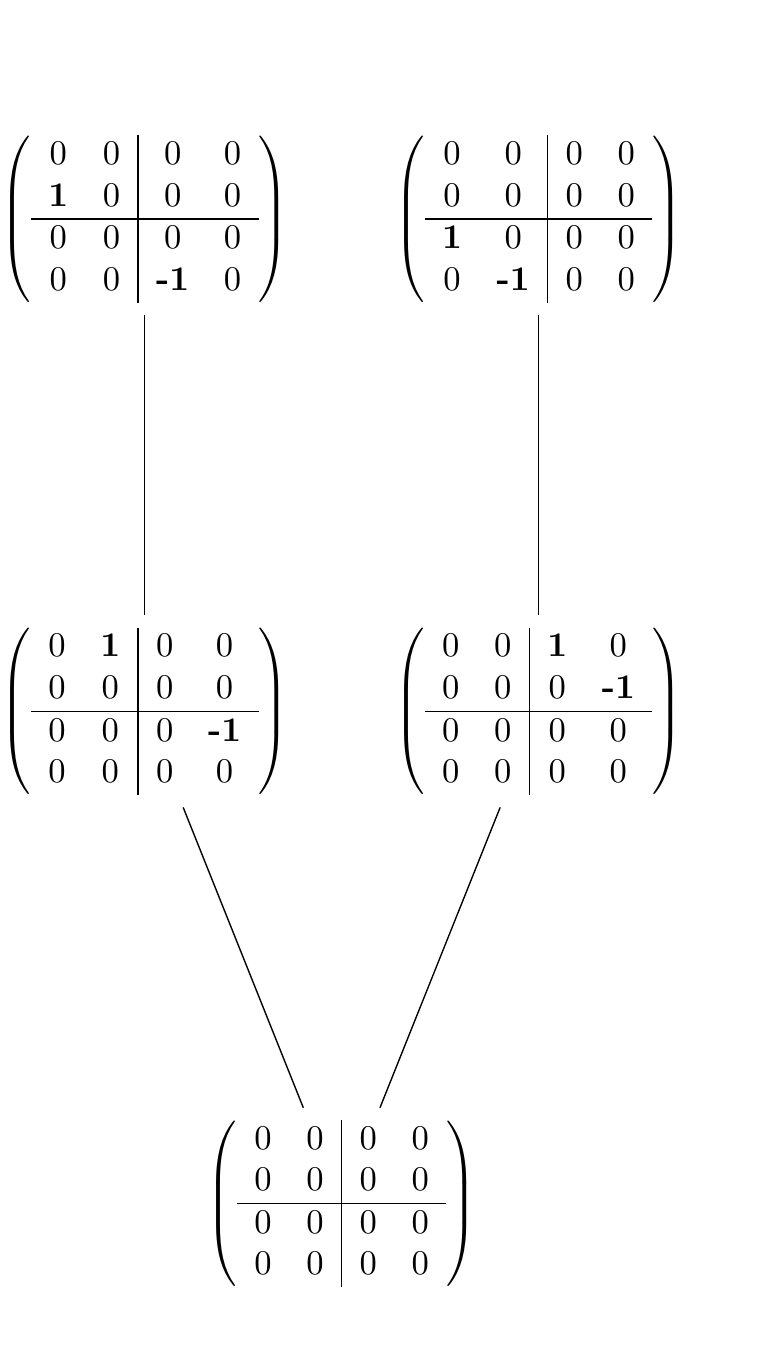}};
  \end{tikzpicture}
\end{center}\caption{$B(\ee)$-orbits and their closures inside $\mathcal{N}^{(2)}(\ee)$}\label{Fig:D}
\end{figure}
\end{example}

\subsection{Interpretation of Proposition~\ref{Prop:Counterex} in terms of root systems}\label{Rem:Papi}
In view of the connection with the Borel orbits, Proposition~\ref{Prop:Counterex} can be proved using the results of \cite{GMP}. 
We are grateful to Paolo Papi for sharing this with us. 
We fix $l\geq 4$ and $n=2l$. We consider the vector space $U=\mathrm{k}^n$ with standard basis $(e_1,\cdots, e_n)$ endowed with the symmetric bilinear form given by $\langle e_i, e_{n+1-j}\rangle=\delta_{i,j}$ for all $i,j=1,\cdots, n$.  For simplicity of notation, for an index $i=1,\cdots, l$ we denote by $i^\ast=n+1-i$ so that so that $\langle e_i, e_{j^\ast}\rangle=\delta_{i,j}=\langle e_{j^\ast}, e_i\rangle$ and $\langle e_i,e_j\rangle=\langle e_{i^\ast},e_{j^\ast}\rangle=0$ for all $i,j=1,\cdots, l$. Given $g\in \GL(\mathrm{k}^n)$ we denote by $g^t$ the adjoint of $g$ with respect to the given form. The orthogonal group $\OO(2l)$ is the group of fixed points for the involution $\rho$ of $\GL(\mathrm{k}^n)$ given by $g\mapsto g^\rho=(g^t)^{-1}$; it consists of two connected components, consisting of matrices of determinant $\pm 1$. The special orthogonal group $\SO(2l)$ is the connected component of the identity and it is a Lie group of type $D_l$. The standard flag $\mathcal{F}_\bullet= U_1\subset\cdots \subset U_n=U$ given by $U_i=\textrm{Span}(e_1,\cdots, e_i)$ is a maximal isotropic flag with respect to the chosen bilinear form and we denote by $\B$ its stabilizer in $\GL(\mathrm{k}^n)$. We notice that  $\B$ is $\rho$-invariant and we denote by $\B(1)\subset B$ the Borel subgroup of $\OO(2l)$ consisting of $\rho$-fixed points. 

We denote by $E_{ij}$ the $2l\times 2l$ matrix having $1$ at place $(i,j)$ and zero elsewhere. The adjoint of $E_{ij}$ with respect to the fixed bilinear form is $E_{j^\ast i^\ast}$ and thus the involution $\Delta$ sends $E_{ij}$ to $-E_{j^\ast i^\ast}$. Let $M_\gamma=E_{l^\ast1}-E_{1^\ast l}$ and $N_\gamma=E_{l1}-E_{1^\ast l^\ast}$.  We see that $M_\gamma$ and $N_\gamma$ are fixed by $\Delta$ and thus are contained in the Lie algebra of $\SO(2l)$. Let $P_{ll^\ast}$ be the elementary matrix obtained by permuting the rows $l$ and $l^\ast$ of the identity matrix. This matrix is contained in $\OO(2l)$ and it has determinant $-1$. It acts as an outer automorphism of $\SO(2l)$. Moreover, $P_{ll^\ast}U_i=U_i$ for $i\neq l$ and $P_{l,l^\ast}U_l=\textrm{Span}(e_1,\cdots, e_{l-1}, e_{l^\ast})=U_l'$. We notice that for every $b\in \B(1)$, $b U_l'=U_l'$ (this is not true for $b\in B$). Putting these observations together we get the following (well-known) fact: for every $b\in \B(1)$
\begin{equation}\label{Eq:PB}
P_{ll^\ast}bP_{ll^\ast}\in \B(1).
\end{equation}
This means that $P_{ll^\ast}$ stabilizes $\B(1)$ (but not $\B$).  Moreover
\begin{equation}\label{Eq:PMN}
P_{ll^\ast}\cdot M_\gamma=P_{ll^\ast}M_\gamma P_{ll^\ast}=N_\gamma. 
\end{equation}
Putting together \eqref{Eq:PB} and \eqref{Eq:PMN} we get 
\[
\B(1)\cdot M_\gamma=P_{ll^\ast}\B(1)P_{ll^\ast}\cdot M_\gamma=P_{ll^\ast}\B(1)\cdot N_\gamma
\]
and thus the $\B(1)$-orbits of $M_\gamma$ and $N_\gamma$ are isomorphic and hence their closures have the same dimension. This implies the second part of Proposition~\ref{Prop:Counterex}. 

To get the first part of Proposition~\ref{Prop:Counterex} (type $A$ situation) we denote by $\ee_i=E_{ii}$ and we consider the Cartan subalgebra $\mathfrak{h}\subset \mathfrak{sl}_n$ of traceless diagonal matrices with basis $(\ee_1-\ee_2,\cdots, \ee_l-\ee_{l^\ast}, \cdots, \ee_{2^\ast}-\ee_{1^\ast})$ and the corresponding root space decomposition $\mathfrak{g}=\mathfrak{sl}_n=\mathfrak{h}\oplus \bigoplus_{\alpha\in \Phi}\mathfrak{g}_\alpha$. Here we choose the simple roots naturally as  $\alpha_1,\cdots, \alpha_{l}=\alpha_{l^\ast},\cdots,\alpha_{1^\ast}$ of $\Phi$ so that $\alpha_i$ corresponds to $\ee_i-\ee_{i+1}$  and $\alpha_{i^\ast}$ corresponds to $\ee_{(i+1)^\ast}-\ee_{i^\ast}$ for $i=1,\cdots, l$. The Dynkin diagram is
\[
\xymatrix{
\circ\ar@{-}_(.0){\al_1}[r]&\circ\ar@{..}_(.0){\al_2}[r]&\circ\ar@{-}_(.0){\al_{l-1}}[r]&\circ\ar@{-}_(.0){\al_l=\al_{l^\ast}}[r]&\circ\ar@{..}_(.0){\al_{(l-1)^\ast}}[r]&\circ\ar@{-}_(.0){\al_{2^\ast}}_(1.0){\al_{1^\ast}}[r]&\circ
}
\]
The remaining roots are $\pm \alpha_{ij}$ where $\alpha_{ij}=\alpha_i+\cdots+\alpha_j$ for $1\leq i\leq j\leq n-1$ (with the convention that $\alpha_{n-i}:=\alpha_{i^\ast}$ for $i=1,\cdots, l$) and $\Phi$ is a root system  of type $A_{n-1}$. Obviously,  $\mathfrak{g}_{\alpha_{i(j-1)}}=\textrm{Span}(E_{ij})$ and  $\mathfrak{g}_{-\alpha_{i(j-1)}}=\textrm{Span}(E_{ji})$. We denote by $x_\alpha$ the generator $E_{ij}$ of $\mathfrak{g}_\alpha$. Thus, 
\[
\begin{array}{cc}
M_\gamma=x_{-\alpha_{1l}}-x_{-\alpha_{l^\ast 1^\ast}}&
N_\gamma=x_{-\alpha_{1(l-1)}}-x_{-\alpha_{(l-1)^\ast1^\ast}}.
\end{array}
\]
We consider the strongly orthogonal sets of roots (see \cite{GMP}) $\mathcal{M}$ and $\mathcal{N}$ corresponding to $M_\gamma$ and $N_\gamma$ respectively: 
\[
\begin{array}{cc}
\mathcal{M}=(-\alpha_{1l}, -\alpha_{l^\ast 1^\ast})&
\mathcal{N}=(-\alpha_{1(l-1)}, -\alpha_{(l-1)^\ast1^\ast})
\end{array}
\]
and their affine analogue considered in \cite{GMP}
\[
\begin{array}{cc}
\widehat{\mathcal{M}}=(-\alpha_{1l}-\delta, -\alpha_{l^\ast 1^\ast}-\delta)&
\widehat{\mathcal{N}}=(-\alpha_{1(l-1)}-\delta, -\alpha_{(l-1)^\ast1^\ast}-\delta)
\end{array}
\]
where $\delta$ denotes the minimal positive imaginary root of the affine root system of type $A_{n-1}$. Following \cite{GMP} we consider the corresponding elements of the affine Weyl group
\[
\begin{array}{cc}
\sigma_{\widehat{\mathcal{M}}}=\sigma_{-\alpha_{1l}-\delta}\sigma_{-\alpha_{l^\ast 1^\ast}-\delta}&
\sigma_{\widehat{\mathcal{N}}}=\sigma_{-\alpha_{1(l-1)}-\delta}\sigma_{-\alpha_{(l-1)^\ast 1^\ast}-\delta}
\end{array}
\]
where $\sigma_\alpha$ denotes the reflection through the affine root $\alpha$. Let us denote $s_i=\sigma_{\alpha_i}$ the simple reflection through the simple root $\alpha_i$. Since $s_l(-\alpha_{1(l-1)}-\delta)=-\alpha_{1l}-\delta$ and $s_l(-\alpha_{(l-1)^\ast 1^\ast}-\delta)=-\alpha_{l^\ast 1^\ast}-\delta$, we get
\begin{equation}\label{Eq:SigmaMSigmaN}
s_l\sigma_{\widehat{\mathcal{N}}}s_l=\sigma_{\widehat{M}}.
\end{equation}
To compute the lenghts  $\ell(\sigma_{\widehat{M}})$ and $\ell(\sigma_{\widehat{\mathcal{N}}})$ we use \cite[Theorem~1]{GMP} and Lemma~\ref{Lem:Stabilizers}: from \cite[Theorem~1]{GMP} we get 
\[
\begin{array}{cc}
\ell(\sigma_{\widehat{M}})=2\textrm{dim} B\cdot M_\gamma-|\mathcal{M}|,&
\ell(\sigma_{\widehat{N}})=2\textrm{dim} B\cdot N_\gamma-|\mathcal{N}|.
\end{array}
\]
By Lemma~\ref{Lem:Stabilizers},  \eqref{Eq:StabilizersIso} and the fact that $\dim B=l(2l+1)$ we get:
\[
\begin{array}{cc}
\dim B\cdot M_\gamma=6(l-1)+1,&\dim B\cdot N_\gamma=6(l-1)
\end{array}
\]
from which it follows that 
\[
\begin{array}{cc}
\ell(\sigma_{\widehat{M}})=12(l-1),&
\ell(\sigma_{\widehat{N}})=12(l-1)-2=\ell(\sigma_{\widehat{M}})-2.
\end{array}
\]
Thus, \eqref{Eq:SigmaMSigmaN} shows that $\sigma_{\widehat{\mathcal{N}}}$ is a (reduced) subexpression of a reduced expression for $\sigma_{\widehat{\mathcal{M}}}$, hence $\sigma_{\widehat{\mathcal{N}}}<\sigma_{\widehat{\mathcal{M}}}$
in the Bruhat order.  By \cite[Theorem~1]{GMP} this implies that
$BN\subset \overline{BM}$
which is another proof of the first part of Proposition~\ref{Prop:Counterex}. 
\begin{remark}
The matrices $M_\gamma$ and $N_\gamma$ belong to the Lie algebra of $\SO(2l)$ but also of $\SP(2l)$ and (by adding a row and a column of zeros) of $\SO(2l+1)$. Thus one can compute the dimensions of the orbits of those two elements for the Borel $B$ of the different groups of type $A$, $B$, $C$ and $D$. They are shown in Table~\ref{Table:BOrbitDimension}, which can be calculated in a similar manner as in Lemma~\ref{Lem:Stabilizers}.
\begin{table}
\begin{tabular}{|c|c|c|c|c|}
\hline
&$\dim\mathrm{Stab}_B(M_\gamma)$&$\dim B.M_\gamma $& $\dim\mathrm{Stab}_B(N_\gamma)$ & $\dim B.N_\gamma$\\\hline
Type A&(2l-1)(l-2)+3&6(l-1)+1  &(2l-1)(l-2)+4& 6(l-1)  \\\hline
Type B&l(l-2)+2&3(l-1)+1 &l(l-2)+3& 3(l-1)  \\\hline
Type C&l(l-2)+1&3(l-1)+2 &l(l-2)+2&  3(l-1)+1 \\\hline
Type D&(l-1)(l-2)+2&3(l-1)-1  &(l-1)(l-2)+2&  3(l-1)-1 \\\hline
\end{tabular}
\caption{Dimension of the Borel orbits and stabilizers}\label{Table:BOrbitDimension}
\end{table}
It is worth noticing that in types $A$, $B$ and $C$, $BN_\gamma\subset\overline{BM_\gamma}$, but not in type $D$.
\end{remark}
It is interesting to reinterprete the $B(1)$-orbits of $M$ and $N$ in terms of root systems. Let us consider the following elements of the root lattice $\mathbb{Z}\Phi$
\[
\beta_1=\alpha_1+\alpha_{1^\ast}, \beta_2=\alpha_2+\alpha_{2^\ast},\cdots, \beta_{l-1}=\alpha_{l-1}+\alpha_{l+1}, \beta_l=\alpha_{l-1}+\alpha_{l+1}+2\alpha_l.
\]
It is straightforward to check that $(\beta_1,\cdots, \beta_l)$ form a a set of simple roots for a root system of type $D_l$ and the corresponding Dynkin diagram is 
\[
\xymatrix@R=3pt{
&&&\circ\\
\circ\ar@{-}_(.0){\beta_1}[r]&\circ\ar@{..}_(.0){\beta_2}[r]&\circ\ar@{-}_(.0){\beta_{l-2}}_(1.0){\beta_{l-1}}[ur]\ar@{-}_(1.0){\beta_l}[dr]&\\
&&&\circ
}
\]
With respect to this root system, $M$ and $N$ become root vectors:
\[
\begin{array}{cc}
M=x_{\beta_1+\beta_2+\cdots+\beta_{l-2}+\beta_{l-1}},& N=x_{\beta_1+\beta_2+\cdots+\beta_{l-2}+\beta_{l}}
\end{array}
\]
which can be represented in the Dynkin diagram as follows:	
\[
\xymatrix@R=3pt{
&&&\circ\\
\circ\ar@{..}@/^1pc/|M[rrru]\ar@{..}@/_1pc/|N[rrrd]\ar@{-}_(.0){\beta_1}[r]&\circ\ar@{..}_(.0){\beta_2}[r]&\circ\ar@{-}_(.0){\beta_{l-2}}_(1.0){\beta_{l-1}}[ur]\ar@{-}_(1.0){\beta_l}[dr]&\\
&&&\circ
}
\]
In particular, one sees  that the outer automorphism $P_{ll^\ast}$  of $SO(2l)$ provides an isomorphism between the two $B(1)$-orbits of $M$ and $N$ which therefore have the same dimension.

\section{Conjectures}\label{Sec:General}

We have shown  that Main Question~\ref{Question2} is not in general true. Under which circumstances it is (not) true, however, is still an open question.  

\begin{conjecture}\label{Conj} Let $\A$ be the Seesaw algebra. Main Question \ref{Question2} has a positive answer restricted to any irreducible component of the symmetric representation variety $R(\A,V)^{\form, \ee}$.  In other words if $I$ is an irreducible component of $R(\A,V)^{\form, \ee}$ and $M,N\in I$ then $\xymatrix@1{M\degg^\ee N \ar@{<=>}[r]& M\degg N}$.
\end{conjecture}

In type $C$ the symmetric representation varieties are irreducible, whereas in type $B$ we are not sure at the moment. In type $D$ there are two dense orbits.
\begin{conjecture}\label{ConjB}
In type  $C$, Borel orbit closures are induced by type $A$. 
\end{conjecture}
This conjecture is also motivated by \cite[Corollary~8.4.9]{BB}.

\subsection*{Acknowledgements} 
We thank Francesco Esposito and Giovanna Carnovale for many useful discussions, in particular concerning Example~\ref{Ex:Seesaw}. We thank Martin Bender for conversations about the contents of this article. We thank Paolo Papi for insights about the root-theoretic interpretation of our results shown in Section~\ref{Rem:Papi}. We thank Salvatore Stella for advises concerning affine Weyl groups.  This work was sponsored by DFG Forschungsstipendium BO 5359/1-1, DFG Rückkehrstipendium BO 5359/3-1  and DFG Sachbeihilfe BO 5359/2-1.

\bibliographystyle{amsplain}

\end{document}